\documentclass[11pt]{article}
\usepackage{graphicx}
\usepackage{amsthm, amsmath, amssymb, tikz, bm}
\usepackage{mathtools}
\usepackage{xifthen}
\usepackage{caption}
\usepackage[colorlinks=true,
linkcolor=blue,citecolor=blue,
urlcolor=blue]{hyperref}

\usepackage[margin=1in]{geometry} 

\usepackage[shortlabels]{enumitem}
\usepackage{todonotes}
\usetikzlibrary{calc,shapes, backgrounds}
\allowdisplaybreaks

\newtheorem{lemma}{Lemma}
\newtheorem{theorem}{Theorem}

\newtheorem{claim}{Claim}

\newtheorem{conjecture}{Conjecture}

\newcommand{\ds}{\displaystyle}
\newcommand{\dss}{\displaystyle\sum}

\newcommand{\lp}{\left(}
\newcommand{\rp}{\right)}

\newcommand{\cA}{\mathcal{A}}

\newcommand{\cF}{\mathcal{F}}
\newcommand{\cH}{\mathcal{H}}

\newcommand{\paren}[1]{\lp#1\rp}%
\newcommand{\norm}[1]{\left\lVert#1\right\rVert}

\let\lk\Gamma
\let\st\Sigma
\def\iv{^{-1}}
\let\bsh\backslash
\def\dist{{\fam0 dist}}
\def\hf{{\frac12}}
\def\pS{{\widetilde S}}

\def\vi{{v_i}}
\def\vipo{{v_{i+1}}}
\def\vz{{v_0}}
\def\uz{{u_0}}

\def\far{\Phi}

\let\br=B
\let\fr=F
\def\fri{\fr_i}

\def\rvi{p_i}
\def\mrvi{{\alpha_i}}
\def\nrvi{{\beta_i}}

\def\ui{{u_i}}
\def\uipo{{u_{i+1}}}

\def\dd{{d'}}
\def\dR{R'}
\def\dbr{{\br'}}
\def\dbri{{\br'_i}}

\title{Maximum spectral radius of outerplanar 3-uniform hypergraphs}

\author{
M.~N.~Ellingham
\thanks{Vanderbilt University, Nashville, TN 37240,
({\tt mark.ellingham@vanderbilt.edu}).  This author was supported in part by Simons Foundation award no.~429625.}
\and
Linyuan Lu
\thanks{University of South Carolina, Columbia, SC 29208,
({\tt lu@math.sc.edu}). This author was supported in part by NSF grant
DMS-1600811 and a Southeastern Conference Visiting Faculty Travel
Grant.} \and
Zhiyu Wang \thanks{Georgia Institute of Technology, Atlanta, GA, 30332,
({\tt zwang672@gatech.edu}).} 
}

\begin{document}

\maketitle

\begin{abstract}
In this paper, we study the maximum spectral radius of outerplanar $3$-uniform hypergraphs. Given a hypergraph $\mathcal{H}$, the shadow of $\mathcal{H}$ is a graph $G$ with $V(G)= V(\mathcal{H})$ and $E(G) = \{uv: uv \in h \textrm{ for some } h\in E(\mathcal{H})\}$.
A graph is \textit{outerplanar} if it can be embedded in the plane such that all its vertices lie on the outer face. A $3$-uniform hypergraph $\mathcal{H}$ is called \textit{outerplanar} if its shadow has an outerplanar embedding such that every hyperedge of $\mathcal{H}$ is the vertex set of an interior triangular face of the shadow. Cvetkovi\'c and Rowlinson conjectured in 1990 that among all outerplanar graphs on $n$ vertices, the graph $K_1+ P_{n-1}$ attains the maximum spectral radius. We show a hypergraph analogue of the Cvetkovi\'c-Rowlinson conjecture. In particular, we show that for sufficiently large $n$, the $n$-vertex outerplanar $3$-uniform hypergraph of maximum spectral radius is the unique $3$-uniform hypergraph whose shadow is $K_1 + P_{n-1}$.
\end{abstract}

\section{Introduction}

A graph $G$ is {\em planar} if it can be embedded in the plane, i.e., it can be drawn on the plane in such a way that edges intersect only at their endpoints. A graph is {\em outerplanar} if it can be embedded in the plane so that all vertices lie on the boundary of its outer face. 
The study of the spectral radius of (outer)planar graphs has a long history, dating back to Schwenk and Wilson \cite{Schwenk-Wilson78}. Given a graph $G$, the \emph{spectral radius} $\lambda$ of $G$ is the largest eigenvalue of the adjacency matrix of $G$. The spectral radius of planar graphs is useful in geography as a measure of the overall connectivity of a planar graph \cite{Boots-Royle91, Cvetkovic-Rowlinson90}. It is therefore of interest to geographers to find the maximum spectral radius of a planar graph as a theoretical upper bound for the connectivity of networks. Boots and Royle \cite{Boots-Royle91}, and independently Cao and Vince \cite{Cao-Vince93} conjectured that the extremal planar graph achieving the maximum spectral radius is $K_2 + P_{n-2}$ (see Figure \ref{fig:extremal}). Hong \cite{Yuan88} first showed that for an $n$-vertex plananr graph $G$, $\lambda(G) \leq \sqrt{5n-11}$. This was subsequently improved in a series of papers \cite{Cao-Vince93, Yuan95, Guiduli96, Yuan98, Ellinghan-Zha00}. Guiduli and Hayes \cite{Guiduli-Hayes98} showed in an unpublished preprint that the Boots-Royle-Cao-Vince conjecture is true for sufficiently large $n$.
For outerplanar graphs, it is conjectured by Cvetkovi\'c and Rowlinson \cite{Cvetkovic-Rowlinson90} that among all outerplanar graph on $n$ vertices, $K_1 + P_{n-1}$ attains the maximum spectral radius (see Figure \ref{fig:extremal}). Partial progress has been made by Rowlinson \cite{Rowlinson90}, Cao and Vince \cite{Cao-Vince93}, and Guiduli and Hayes \cite{Guiduli-Hayes98}.
Recently, Tait and Tobin \cite{Tait-Tobin17} proved the Boots-Royle-Cao-Vince conjecture and the Cvetkovi\'c-Rowlinson conjecture for large enough $n$. Lin and Ning \cite{Lin-Ning19} showed that the Cvetkovi\'c-Rowlinson conjecture holds for all $n\geq 2$ except for $n = 6$.

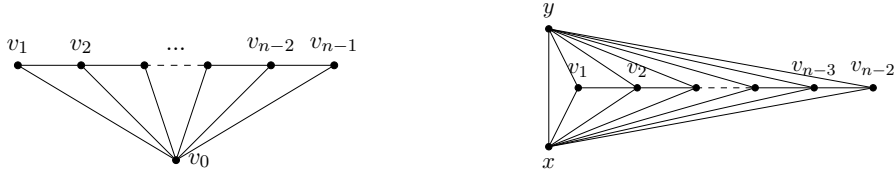
\begin{figure}[htb]
\hbox to \hsize{
	\hfil
	\resizebox{5cm}{!}{\begin{tikzpicture}[scale=1, Wvertex/.style={circle, draw=black, fill=white, scale=1}, bvertex/.style={circle, draw=black, fill=black, scale=0.3}]

\node [bvertex, label=right:$v_0$] (v0) at (-0.5, -0.5) {};
\node [bvertex, label=above:$v_1$] (v1) at (-3, 1) {};
\node [bvertex, label=above:$v_2$] (v2) at (-2, 1) {};
\node [bvertex] (v3) at (-1, 1) {};
\node [bvertex] (v4) at (0, 1) {};
\node [bvertex, label=above:$v_{n-2}$] (v5) at (1, 1) {};
\node [bvertex, label=above:$v_{n-1}$] (v6) at (2, 1) {};

\draw (v0) -- (v1);
\draw (v0) -- (v2);
\draw (v0) -- (v3);
\draw (v0) -- (v4);
\draw (v0) -- (v5);
\draw (v0) -- (v6);

\draw (v1) -- (v2);
\draw (v3) -- (v2); 
\draw [dashed] (v4) -- (v3) node [midway, fill=white, above=3pt] {$...$};
\draw (v5) -- (v4);
\draw (v6) -- (v5);

\end{tikzpicture}}%
	\hfil
	\resizebox{5cm}{!}{\begin{tikzpicture}[scale=1, Wvertex/.style={circle, draw=black, fill=white, scale=1}, bvertex/.style={circle, draw=black, fill=black, scale=0.3}]

\node [bvertex, label=below:$x$] (x) at (-2.5, -1) {};
\node [bvertex, label=above:$y$] (y) at (-2.5, 1) {};
\node [bvertex, label=above:$v_1$] (v1) at (-2, 0) {};
\node [bvertex, label=above:$v_2$] (v2) at (-1, 0) {};
\node [bvertex] (v3) at (0, 0) {};
\node [bvertex] (v4) at (1, 0) {};
\node [bvertex, label=above:$v_{n-3}$] (v5) at (2, 0) {};
\node [bvertex, label=above:$v_{n-2}$] (v6) at (3, 0) {};

\draw (x) -- (y);
\draw (x) -- (v1);
\draw (x) -- (v2);
\draw (x) -- (v3);
\draw (x) -- (v4);
\draw (x) -- (v5);
\draw (x) -- (v6);

\draw (y) -- (v1);
\draw (y) -- (v2);
\draw (y) -- (v3);
\draw (y) -- (v4);
\draw (y) -- (v5);
\draw (y) -- (v6);

\draw (v1) -- (v2);
\draw (v3) -- (v2); 
\draw [dashed] (v4) -- (v3);
\draw (v5) -- (v4);
\draw (v6) -- (v5);

\end{tikzpicture}}%
	\hfil
}
    \caption{The graphs $K_1 + P_{n-1}$ (left) and $K_2 + P_{n-2}$ (right).}
    \label{fig:extremal}
\end{figure}

In this paper, we extend the investigations into the maximum spectral radius of outerplanar $3$-uniform hypergraphs. Given a hypergraph $\cH$, the \textit{shadow} of $\cH$, denoted by $\partial(\cH)$, is a $2$-uniform graph $G$ with $V(G) = V(\cH)$ and $E(G) = \{uv: uv \in h \textrm{ for some } h\in E(\cH)\}$.

We adopt Zykov's \cite{Zykov74} definition of hypergraph planarity. In particular, a $3$-uniform hypergraph $\cH$ is called {\em planar} if $\partial(\cH)$ has a planar embedding so that every hyperedge of $\cH$ is the vertex set of a triangular face of $\partial(\cH)$. A $3$-uniform hypergraph $\cH$ is called {\em outerplanar} if $\partial(\cH)$ has an outerplanar embedding such that every hyperedge of $\cH$ is the vertex set of an interior triangular face of $\partial(\cH)$.

Now we define the spectral radius of an $r$-uniform hypergraph. Given positive integers $r$ and $n$, an order $r$ and dimension $n$ \textit{tensor} $\cA = (a_{i_1 i_2 \cdots i_r})$ over $\mathbb{C}$ is a multidimensional array with all entries $a_{i_1 i_2 \cdots i_r} \in \mathbb{C}$ for all $i_1, i_2, \cdots, i_r \in [n] = \{1,2,\dots,n\}$.
Given a column vector $\bm{x} = (x_1, x_2, \cdots, x_n)^T \in \mathbb{C}^n$, $\cA \bm{x}^{r-1}$ is defined to be a vector in $\mathbb{C}^n$ whose $i$th entry is 
$$(\cA \bm{x}^{r-1})_i = \dss_{i_2, \cdots,i_r= 1}^n a_{i i_2 \cdots i_r} x_{i_2} \cdots x_{i_r}.$$
In 2005, Qi \cite{Qi05} and Lim \cite{Lim05} independently proposed the definition of eigenvalues of a tensor. In particular, if there exists a number $\lambda \in \mathbb{C}$ and a nonzero vector $\bm{x} \in \mathbb{C}^n$ such that 
$$\cA \bm{x}^{r-1} = \lambda \bm{x}^{[r-1]}$$
where $\bm{x}^{[r-1]} = (x_1^{r-1}, x_2^{r-1}, \cdots, x_n^{r-1})^T$, then $\lambda$ is called the \textit{eigenvalue} of $\cA$ and $\bm{x}$ is called an \textit{eigenvector} of $\cA$ corresponding to $\lambda$. The \textit{spectral radius} of $\cA$, denoted by $\lambda(\cA)$, is the maximum modulus of the eigenvalues of $\cA$. It was shown in \cite{Qi13} that 
$$\lambda(\cA) = \max_{\substack{||\bm{x}||_r=1\\
\bm{x} \in \mathbb{R}_{+}^n}} \bm{x}^T \cA \bm{x}^{r-1},$$
where $||\bm{x}||_r:=(|x_1|^r+|x_2|^r+\cdots+|x_n|^r)^{1/r}$ and $\mathbb{R}_{+}$ is the set of nonnegative real numbers. 

In 2012, Cooper and Dutle \cite{Cooper-Dutle12} defined the \textit{adjacency tensor} of an $r$-uniform hypergraph. Given an $r$-uniform hypergraph $\cH$ on $n$ vertices, the adjacency tensor $\cA(\cH)$ of $\cH$ is defined as the order $r$ dimension $n$ tensor with entries $a_{i_1 i_2 \cdots i_r}$ such that 
\begin{equation*}
    a_{i_1 i_2 \cdots i_r} = \begin{cases}
                                    \frac{1}{(r-1)!} & \textrm{if } \{i_1, i_2, \cdots i_r\} \in E(\cH)\\
                                    0 & \textrm{otherwise.}
                             \end{cases}
\end{equation*}
Let $\lambda(\cH)$ denote the spectral radius of $\cA(\cH)$. Given an $r$-uniform hypergraph $\cH$ and a vector $\bm{x}=(x_1,x_2,\ldots,x_n) \in\mathbb{R}^n$, we can define a multi-linear function $P_{\cH}(\bm{x}): \mathbb{R}^n \to \mathbb{R}$ as follows:
\[
P_H(\bm{x})=r\sum_{\{i_1,i_2,\ldots,i_r\}\in E(\cH)}x_{i_1}x_{i_2}\cdots x_{i_r}.
\] Then the spectral radius of $\cH$ can be also expressed as
\begin{equation*}
\lambda(\cH):=\max_{||\bm{x}||_r=1}P_{\cH}(\bm{x}) = \max_{\bm{x}\in \mathbb{R}_{+}^n} \frac{P_{\cH}(\bm{x})}{\norm{\bm{x}}_r^r}.
\end{equation*}

The Perron-Frobenius theorem \cite{CPZ08, FGH13} for nonnegative tensors implies that there is always a nonnegative vector $\bm{x}$ satisfying the maximum at right above.  Any such $\bm{x}$ is called a \textit{Perron-Frobenius eigenvector} of $\cA(\cH)$ (corresponding to $\lambda(\cH)$).
If $\cH$ is connected then a Perron-Frobenius eigenvector is strictly positive and is unique up to scaling by a positive coefficient; moreover, the spectral radius $\lambda(\cH)$ is the unique eigenvalue with a strictly positive eigenvector.
By definition, the spectral radius $\lambda(\cH)$ and its eigenvector $\bm{x} = (x_1, \cdots,x_n)$ also satisfy the following \textit{eigenequation} for every $x_i$:

\begin{equation*}
\lambda(H)x_i^{r-1} = \sum_{\{i,i_2,\ldots,i_r\}\in E(H)}x_{i_2}\cdots x_{i_r}
~~\text{for}\ x_i>0.
\end{equation*}

Now we are ready to state our main theorem. We use $\cF_n$ to denote the \textit{fan hypergraph}, i.e., the unique $3$-uniform hypergraph whose shadow is $K_1 + P_{n-1}$ (see Figure \ref{fig:extremal}).

\begin{theorem}\label{thm:outer}
For large enough $n$, the $n$-vertex outerplanar $3$-uniform hypergraph of maximum spectral radius is the fan hypergraph $\cF_n$.
\end{theorem}

The shadow of the extremal hypergraph attaining the maximum spectral radius among all outerplane $3$-uniform hypergraphs is exactly the extremal graph attaining the maximum spectral radius among all outerplanar graphs. This motivates us to make the following analogous conjecture for planar $3$-uniform hypergraphs:

\begin{conjecture}\label{thm:planar}
For large enough $n$, the $n$-vertex planar $3$-uniform hypergraph graph $\cH$ of maximum spectral radius is the unique maximal hypergraph whose shadow is $K_2 + P_{n-2}$. 
\end{conjecture}

\section{Proof of Theorem \ref{thm:outer}}

Let $\cH$ be an $n$-vertex outerplanar $3$-uniform hypergraph of maximum spectral radius.
Let $G$ be the shadow of $\cH$, i.e., $V(G) = V(\cH)$ and $E(G) = \{vu: \{v,u\}\subseteq h \textrm{ for some } h\in E(\cH) \}$. It follows by definition that $G$ is outerplanar, and thus does not contain a $K_{2,3}$ minor or a $K_4$ minor.
Observe that $\cH$ must be edge-maximal (while maintaining the outerplanarity). Otherwise, we can obtain an outerplanar hypergraph $\cH'$ such that $\cH\subsetneq \cH'$. It then follows from the Perron-Frobenius Theorem that $\cH'$ attains a larger spectral radius than $\cH$, giving us a contradiction. Now since $\cH$ is edge-maximal, $G$ must be a maximal outerplanar graph, with $2n-3$ edges.
Then $G$ is $2$-connected, and has an outerplanar embedding, unique up to homeomorphisms of the plane, whose outer face is bounded by a Hamilton cycle.  We always assume $G$ has this outerplanar embedding.
All interior faces of $G$ are triangles, and every triangle of $G$ is a facial triangle and a hyperedge of $\cH$.
The dual of $G$ (excluding the outer face) is a tree, so the interior faces of $G$ are connected together in a treelike fashion.

We use $N(v)$ to denote the set of neighbors of $v$ in $G$, i.e., $N(v) = \{u: vu\in E(G)\}$ and $d(v)$ to denote the degree of $v$, i.e., $d(v)=|N(v)|$. We also use $d_F(v)$ to refer to degree in a subgraph $F$ of $G$.  The closed neighborhood of $v$, denoted by $N[v]$, is defined as $N[v] = N(v)\cup \{v\}$.
More generally, we let $\dist(u,v)$ denote the distance between $u$ and $v$ in $G$, and $N_k(v) = \{ u \in V(G) : \dist(v, u) = k\}$.
Given an edge $uw$ and vertex $v$ define the \emph{level of $uw$ relative to $v$} to be $(\dist(u, v) + \dist(w, v))/2$, which is an integer or half-integer.

Let $\lk(v) = \{uw: \{v,u,w\}\in E(\cH)\}$ be the link of $v$ in $\cH$, and $d_\cH(v) = |\lk(v)|$ be the degree of $v$ in $\cH$.
The edges in $\lk(v)$ form an induced path in $G$ whose ends are the neighbors of $v$ on the outer cycle.
For each edge $e$ of $G$, $\lk\iv(e)$ is the set of vertices forming a triangle with $e$, and contains one vertex if $e$ is an outer edge, and two vertices otherwise.
We also use $\st(v)$ to denote the set of edges incident with $v$ in $G$.  In our situation the edges in $\st(v)$ and $\lk(v)$ are precisely the edges at levels $\hf$ and $1$, respectively, relative to $v$.

Suppose we are given an edge $uw$ and a vertex $v$ not incident with $uw$.  If $uw$ is an outer edge, define $\far(uw, v)$ to be the empty graph.  Otherwise, $G\bsh\{u,w\}$ has two components; define $\far(uw, v)$ to be the component not containing $v$, together with all edges joining that component to $u$ or $w$.  Loosely, $\far(uw, v)$ is the subgraph of $G$ on the far side of $uw$ from $v$.

\begin{lemma}\label{lem:spec_wheel}
$\lambda(\cH) \geq \sqrt[3]{4(n-1)}\paren{1-\frac{1}{n-1}}$.
\end{lemma}

\begin{proof}

Let $\cF_n$ be the fan hypergraph on $n$ vertices, i.e., the unique $3$-uniform hypergraph on $n$ vertices whose shadow is $K_1 + P_{n-1}$. Suppose $w$ is the vertex that is adjacent to all the other vertices in $\partial(\cF_n)$ and $v_1, v_2, \cdots,v_{n-1}$ are its neighbors. Clearly $\cF_n$ is outerplanar.
Consider the vector $\bm{x} \in \mathbb{R}^n$ with $x_{w} = 1/\sqrt[3]{3}$ and $x_\vi = \paren{\frac{2}{3(n-1)}}^{1/3}$. Note that $\norm{\bm{x}}_3 =1$.
It follows that 
$$\lambda(\cH) \geq \lambda(\cF_n) \geq P_{\cF_n}(\bm{x}) = 3(n-2) \cdot \frac{1}{\sqrt[3]{3}} \cdot \paren{\frac{2}{3(n-1)}}^{2/3} = \sqrt[3]{4(n-1)}\paren{1-\frac{1}{n-1}}. $$
\end{proof}

Note that since $\cH$ is connected, there exists an eigenvector corresponding to $\lambda(\cH)$ such that
all its entries are strictly positive.
In the rest of this section, for convenience we assume that this Perron-Frobenius eigenvector $x$ of $\cH$ is re-normalized so that the maximum eigenvector entry is $1$. Let $\vz$ be the vertex with the maximum eigenvector entry, so that $x_\vz = 1$.
We also define $\uz$ to be a vertex with the second largest eigenvector entry, i.e., $x_\uz = \max_{v \ne \vz} x_v$.
We abbreviate $\lambda(\cH)$ to $\lambda$, and the eigenequation of $\cH$ tells us that
$\lambda x_v^2 = \sum_{uw \in \lk(v)} x_u x_w$ for every vertex $v$.

The following lemma says that $\cH$ is very close to the fan hypergraph $\cF_n$.

\begin{lemma}\label{lem:weak_big_degree}
We have $\lambda=(1+o(1))\sqrt[3]{4n}$ and
$d_G(\vz) \ge n- O(n^{2/3})$. Moreover, for any other vertex $u \neq \vz$, $x_{u} = O(n^{-1/3})$.
\end{lemma}

 We first show a weaker version of Lemma \ref{lem:weak_big_degree}. In particular, we show the following claim.

\begin{claim}\label{cl:weak_big_degree}
$d_G(\vz) \ge n-O(n^{5/6})$.
\end{claim}

\begin{proof}[Proof of Claim \ref{cl:weak_big_degree}]
Let $x$ and $\vz$ be as described above, so that $x_\vz=1$. 
Let $d=d(\vz)$ and suppose that $\lk(\vz)$ forms the path $v_1 v_2 \dots v_d$, where $v_1, v_2, \dots, v_d$ are in clockwise order around $\vz$.
Now by the eigenequation for $x_\vz$,
    \[\lambda = \lambda x_\vz^2
            = \dss_{i=1}^{d -1} x_\vi x_\vipo
        \leq \dss_{i=1}^{d} x_\vi^2, \]
using the fact $ab\leq (a^2+b^2)/2$. Set $z = \dss_{i=1}^{d} x_\vi^2$. We have $\lambda \leq z$. It again follows from the eigenequation expansion that
\begin{align}
    \lambda z &= \dss_{i=1}^{d} \lambda x_\vi^2
		= \dss_{i=1}^d \sum_{vw \in \lk(\vi)} x_v x_w
			\nonumber\\
        &\leq  2 \dss_{i=1}^{d} x_\vz x_\vi + 
		\dss_{i=1}^{d} \dss_{vw \in \lk(\vi)\bsh\st(\vz)}
		x_{v} x_{w} \label{eq:lz1}
\end{align}
Define $\br$ to be the subgraph consisting of the edges in $(\bigcup_{i=1}^d \lk(\vi))\bsh\st(\vz)$ and their endvertices. 
The edges of $\br$ are at levels $1\hf$ and $2$ relative to $\vz$.
For $i \in [d-1]$ let $\fri = \far(\vi \vipo, \vz)$ and $\br_i = \br \cap \fri$.
Fig.~\ref{fig:nbr} shows $\br_i$, indicated in bold (and red, if color is visible).
From \eqref{eq:lz1}, using $x_\vz=1$ and the Cauchy-Schwarz inequality, we have
 \begin{align}
    \lambda z &\le  2 \dss_{i=1}^{d} x_\vi + \dss_{vw \in E(\br)} x_{v} x_{w}
		\nonumber\\
        &\le  2 \sqrt{dz} + \dss_{vw \in E(\br)} x_{v} x_{w}.
		 \label{eq:lz2}
\end{align}

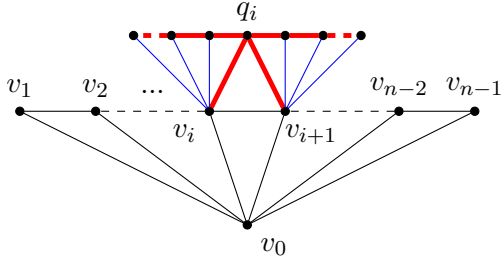
\begin{figure}
	\begin{center}
      \resizebox{7cm}{!}{\begin{tikzpicture}[scale=1, Wvertex/.style={circle, draw=black, fill=white, scale=1}, bvertex/.style={circle, draw=black, fill=black, scale=0.3}]

\node [bvertex, label=below right:$v_0$] (v0) at (-0.5, -0.5) {};
\node [bvertex, label=above:$v_1$] (v1) at (-3.5, 1) {};
\node [bvertex, label=above:$v_2$] (v2) at (-2.5, 1) {};
\node [bvertex, label=below left:$v_i$ ] (v3) at (-1, 1) {};
\node [bvertex, label=below right:{\kern-5pt$v_{i+1}$} ] (v4) at (0, 1) {};
\node [bvertex, label=above:$v_{n-2}$] (v5) at (1.5, 1) {};
\node [bvertex, label=above:$v_{n-1}$] (v6) at (2.5, 1) {};

\node [bvertex,label=above:$q_i$] (v7) at (-0.5, 2) {};
\node [bvertex] (v8) at (-1.5, 2) {};
\node [bvertex] (v9) at (0.5, 2) {};

\node [bvertex] (v10) at (-1, 2) {};
\node [bvertex] (v11) at (0, 2) {};

\node [bvertex] (v12) at (-2, 2) {};
\node [bvertex] (v13) at (1, 2) {};

\draw (v0) -- (v1);
\draw (v0) -- (v2);
\draw (v0) -- (v3);
\draw (v0) -- (v4);
\draw (v0) -- (v5);
\draw (v0) -- (v6);

\draw (v1) -- (v2);
\draw[dashed] (v3) -- (v2) node [midway, fill=white, above=3pt] {$...$};
\draw (v4) -- (v3);
\draw[dashed] (v5) -- (v4);
\draw (v6) -- (v5);


\draw[line width=0.65mm, red] (v7) -- (v3);
\draw[line width=0.65mm, red] (v7) -- (v4);
\draw[line width=0.65mm, red] (v7) -- (v10);
\draw[line width=0.65mm, red] (v10) -- (v8);
\draw[line width=0.65mm, red] (v7) -- (v11);
\draw[line width=0.65mm, red] (v11) -- (v9);
\draw[blue] (v3) -- (v8);
\draw[blue] (v4) -- (v9);
\draw[blue] (v4) -- (v11);
\draw[blue] (v3) -- (v10);

\draw[line width=0.65mm, red, dashed] (v9) -- (v13);
\draw[line width=0.65mm, red, dashed] (v8) -- (v12);
\draw[blue] (v4) -- (v13);
\draw[blue] (v3) -- (v12);

\end{tikzpicture}}
    \end{center} 
    \caption{Neighborhood of $\vz$, and $\br_i$}
    \label{fig:nbr}
\end{figure}

For ease of reference, set $R = \sum_{vw \in E(\br)} x_{v} x_{w}$.
Dividing both sides of inequality \eqref{eq:lz2} by $\lambda$, we have
$z- \frac{2\sqrt{d z}}{\lambda} \leq \frac{R}{\lambda}$. By completing the square, we have $\lp \sqrt{z} -\frac{\sqrt{d}}{\lambda} \rp^2 \leq \frac{R}{\lambda} + \frac{d}{\lambda^2}$. Rearranging the terms of the inequality, we obtain that 
\begin{align}
    z &\leq \paren{\frac{\sqrt{d}}{\lambda} + \sqrt{\frac{d}{\lambda^2}+\frac{R}{\lambda}}  }^2 \nonumber \\
      &= \frac{4d}{\lambda^2} + \frac{2R}{\lambda} - \paren{\sqrt{\frac{d}{\lambda^2}+\frac{R}{\lambda}} -\frac{\sqrt{d}}{\lambda}}^2. \label{eq:z-algebra}
\end{align}
It follows that 
\begin{equation}\label{eq:o1}
    \lambda^3 \leq \lambda^2 z \leq 4d + 2\lambda R - \paren{\sqrt{d + R\lambda}-\sqrt{d}}^2.
\end{equation}
By Lemma \ref{lem:spec_wheel}, we obtain that $\lambda^3 \geq 4n - 16$ when $n$ is large enough. 

Now we find a bound on $2\lambda R$.  Using $ab \le (a^2+b^2)/2$ and then the eigenequations, twice, we have
\begin{align}
2\lambda R &= \dss_{vw \in E(\br)} 2 \lambda x_{v} x_{w}
	\nonumber\\
  &\le \dss_{vw\in E(\br)} \lambda (x_{v}^2 + x_{w}^2)
    = \dss_{u \in V(\br)} d_\br(u)\, \lambda x_u^2
	\nonumber\\
  &= \dss_{u \in V(\br)\cap N_1(\vz) } d_\br(u)\, \lambda x_u^2 
    + \dss_{u \in V(\br)\cap N_2(\vz) } d_\br(u)\, \lambda x_u^2
	\nonumber\\
  &\le 2\dss_{i=1}^d \lambda x_\vi^2 
    + \dss_{u \in V(\br)\cap N_2(\vz) } d_\br(u)\,\lambda x_u^2 
	\nonumber\\
  &= 2\lambda z
    + \dss_{u \in V(\br)\cap N_2(\vz) } d_\br(u)\,
	\sum_{vw \in \lk(u)} x_v x_w.
	\label{eq:tlr1}
\end{align}
If $x_v x_w$ appears in the sum above then the level of $vw$ relative to $\vz$ is between $1$ and $3$.

To investigate the sum in \eqref{eq:tlr1} we examine the structure of $\fri$ and $\br_i$ more closely.  If $\fri$ is nonempty then it contains a common neighbor $q_i$ of $\vi$ and $\vipo$. The vertices of $N_2(\vz) \cap \fri$ lie on a path $\rvi^1 \rvi^2 \rvi^3 \dots \rvi^\nrvi$, with $q_i = \rvi^\mrvi$ for some $\mrvi$ with $1 \le \mrvi \le \nrvi$.
Here $\rvi^j$ is adjacent to $\vi$ for $1 \le j \le \mrvi$, and to $\vipo$ for $\mrvi \le j \le \nrvi$.
The subgraph $\br_i$ contains the edges of this path and edges $\vi q_i, \vipo q_i$.
We let $\fri^j = \far(\rvi^j \rvi^{j+1}, \vz)$ be the part of $\fri$ above $\rvi^j \rvi^{j+1}$ for $j \in [\nrvi-1]$.  See Fig.~\ref{fig:multiplicity5}, which illustrates part of $\fri$ and the edge coefficients in the sum from \eqref{eq:tlr1} ($\vi \vipo$ is dashed because it is not an edge of $\fri$).
For $u = \rvi^j$ with $j \notin \{1,\mrvi,\nrvi\}$ we have $d_\br(u)=2$; for $j \in \{1, \mrvi, \nrvi\}$ the value of $d_\br(u)$ depends on whether $\mrvi=1$ or $\mrvi=\nrvi$ or both, but we always have $1 \le d_\br(u) \le 4$, and $d_\br(u) > 2$ only if $j=\mrvi$.

\begin{figure}
	\begin{center}
		\resizebox{10cm}{!}{


\def\vequals{\rotatebox{90}{$=$}}

\begin{tikzpicture}[scale=1, Wvertex/.style={circle, draw=black, fill=white, scale=1}, bvertex/.style={circle, draw=black, fill=black, scale=0.3}]

\node [bvertex, label={[font=\scriptsize] below left:$v_{i}$} ] (v1) at (-1.2, 0.4) {};
\node [bvertex, label={[font=\scriptsize] below right:$v_{i+1}$}] (v2) at (0.2, 0.4) {};


\node [bvertex] (u0) at (-0.5, 2) {};
  \node [font=\scriptsize] at (-0.5, 1.5) {$q_i$};
  \node [font=\scriptsize] at (-0.5, 1.2) {$\vequals$};
  \node [font=\scriptsize] at (-0.5, 0.85) {$\rvi^\mrvi$};

\node [bvertex] (u1) at (-1.1, 2) {};
\node [bvertex] (u2) at (0.1, 2) {};

\node [bvertex] (u3) at (-1.7, 2) {};
\node [bvertex] (u4) at (0.7, 2) {};

\node [bvertex, label={[font=\scriptsize] above:$\rvi^3$}]
	(u5) at (-2.3, 2) {};
\node [bvertex] (u6) at (1.3, 2) {};

\node [bvertex, label={[font=\scriptsize] above:$\rvi^2$}]
	(u7) at (-2.9, 2) {};
\node [bvertex, label={[font=\scriptsize] above:{\kern13pt$\rvi^{\nrvi\!-\!1}$}}]
	(u8) at (1.9, 2) {};

\node [bvertex, label={[font=\scriptsize] above left:$\rvi^1$}] (u9) at (-3.5, 2) {};
\node [bvertex, label={[font=\scriptsize] above right:$\rvi^\nrvi$}] (u10) at (2.5, 2) {};

\node [bvertex] (w1) at (-0.25, 3.0) {};
\node [bvertex] (w2) at (-0.75, 3.0) {};
\node [bvertex] (w3) at (-1.25, 3.0) {};

\draw [dashed] (v1) -- (v2) node [pos=0.5, below=1pt, scale=0.5, color=black] {$4$};
---------------------------------
\draw[green!75!black] (w1) -- (u2) node [pos=0.4, right=0pt, scale=0.5, color=black] {$4$};

\draw[green!75!black] (w1) -- (u0) node [pos=0.4, left=0pt, scale=0.5, color=black] {$2$};
\draw[green!75!black] (w2) -- (u0) node [pos=0.4, left=0pt, scale=0.5, color=black] {$0$};
\draw[green!75!black] (w3) -- (u0) node [pos=0.4, left=0pt, scale=0.5, color=black] {$0$};
\draw[green!75!black] (w1) -- (w2) node [pos=0.5, above=0pt, scale=0.5, color=black] {$4$};
\draw[green!75!black, dashed] (w3) -- (w2) node [pos=0.5, above=0pt, scale=0.5, color=black] {$4$};

\draw[line width=0.65mm, red] (u0) -- (v1) node [pos=0.5, left=1pt, scale=0.5] {$2$};
\draw[line width=0.65mm, red] (u0) -- (v2) node [pos=0.5, right=1pt, scale=0.5] {$2$};

\draw[blue] (v1) -- (u1) node [pos=0.7, left=1pt, scale=0.5] {$6$};
\draw[blue] (v1) -- (u3) node [pos=0.7, left=1pt, scale=0.5] {$4$};
\draw[blue] (v1) -- (u5) node [pos=0.7, left=1.3pt, scale=0.5] {$4$};
\draw[blue] (v1) -- (u7) node [pos=0.7, left=1.7pt, scale=0.5] {$3$};
\draw[blue] (v1) -- (u9) node [pos=0.7, left=3pt, scale=0.5] {$2$};

\draw[blue] (v2) -- (u2) node [pos=0.7, right=1pt, scale=0.5] {$6$};
\draw[blue] (v2) -- (u4) node [pos=0.7,  right=1pt, scale=0.5] {$4$};
\draw[blue] (v2) -- (u6) node [pos=0.7, right=1.3pt, scale=0.5] {$4$};
\draw[blue] (v2) -- (u8) node [pos=0.7, right=1.7pt, scale=0.5] {$3$};
\draw[blue] (v2) -- (u10) node [pos=0.7, right=3pt, scale=0.5] {$2$};

-----------------------------------

\draw[line width=0.65mm, red] (u0) -- (u1) node [midway, below=1pt, scale=0.5] {$0$};
\draw[line width=0.65mm, red] (u0) -- (u2) node [midway, below=1pt, scale=0.5] {$0$};

\draw[line width=0.65mm, red] (u2) -- (u4) node [midway,  above=1pt, scale=0.5] {$0$};
\draw[line width=0.65mm, red] (u1) -- (u3) node [midway,  above=1pt, scale=0.5] {$0$};

\draw[line width=0.65mm, red, dashed] (u4) -- (u6) node [midway, above=1pt, scale=0.5] {$0$};
\draw[line width=0.65mm, red, dashed] (u3) -- (u5) node [midway, above=1pt, scale=0.5] {$0$};

\draw[line width=0.65mm, red] (u6) -- (u8) node [midway, above=1pt, scale=0.5] {$0$};
\draw[line width=0.65mm, red] (u5) -- (u7) node [midway, above=1pt, scale=0.5] {$0$};
  
\draw[line width=0.65mm, red] (u7) -- (u9) node [midway, above=1pt, scale=0.5] {$0$};
\draw[line width=0.65mm, red] (u8) -- (u10) node [midway, above=1pt, scale=0.5] {$0$};

\end{tikzpicture}}
    \end{center}
    \caption{Structure of $\fri$ and $\br_i$, with edge coefficients}
    \label{fig:multiplicity5}
\end{figure}

Each edge $\vi \vipo \in \lk(\vz)$ occurs in the sum in \eqref{eq:tlr1} only as part of $\lk(q_i)$, when $q_i=\rvi^\mrvi$ exists, so the coefficient of $x_\vi x_\vipo$ is at most $4$.  Thus, the contribution from $\lk(\vz)$, using the eigenequation for $x_\vz$, is
\begin{equation}\label{eq:mul1}
   S_0 = \dss_{u \in V(\br)\cap N_2(\vz) } d_\br(u)\,
	\sum_{vw \in \lk(u)\cap \lk(\vz)} x_v x_w
\le 4\sum_{i=1}^{d-1} x_\vi x_\vipo = 4\lambda x_\vz^2 = 4\lambda .
\end{equation}

Assuming $\fri$ is nonempty, the part of the sum in \eqref{eq:tlr1} coming from $vw \in E(\fri)$, $i \in [d-1]$, is
\begin{equation*}
   S_i = \dss_{u \in V(\br_i)\cap N_2(\vz) } d_\br(u)\,
	\sum_{vw \in \lk(u)\bsh\lk(\vz)} x_v x_w
\end{equation*}
We estimate $S_i$ by computing the sum of the coefficients, i.e., 
$$\pS_i = \dss_{u \in V(\br_i)\cap N_2(\vz) } d_\br(u)\,|\lk(u)\bsh\lk(\vz)|
= \dss_{j=1}^\nrvi d_\br(\rvi^j)\,|\lk(\rvi^j)\bsh\lk(\vz)| .$$
Let $d^-(\rvi^j)$ be the degree of $\rvi^j$ in $\fri^{j-1}$ (or $0$ if $j=1$), and $d^+(\rvi^j)$ be its degree in $\fri^j$ (or $0$ if $j=\nrvi$).
For $j \notin \{1, \mrvi, \nrvi\}$ we have
$$|\lk(\rvi^j)\bsh\lk(\vz)| = |\lk(\rvi^j)| = d(\rvi^j)-1
 = (d^-(\rvi^j) + d^+(\rvi^j) + 3)-1
 = d^-(\rvi^j) + d^+(\rvi^j) + 2.$$
For $j = \mrvi$, if $j \ne 1, \nrvi$ we have
$$|\lk(\rvi^j)\bsh\lk(\vz)| = |\lk(\rvi^j)|-1 = d(\rvi^j)-2
 = (d^-(\rvi^j) + d^+(\rvi^j) + 4)-2
 = d^-(\rvi^j) + d^+(\rvi^j) + 2.$$
If $j=1$ we must reduce the above values by $1$ since there is no edge $\vi \rvi^{j-1}$, and similarly if $j=\nrvi$ we must reduce these values by $1$ since there is no edge $\vipo \rvi^{j+1}$.  These reductions are independent, and do not depend on whether $\mrvi=1$ or $\mrvi=\nrvi$ or both.  Therefore,
$$\pS_i = \sum_{j=1}^\nrvi d_\br(\rvi^j) \left(d^-(\rvi^j) + d^+(\rvi^j) + 2\right) - d_\br(\rvi^1) - d_\br(\rvi^\nrvi) .$$

We are going to compare $\pS_i$ to $|E(\fri)|$.  The number of edges in $\fri$ at levels $1\hf$ and $2$ relative to $\vz$ is just $2\nrvi$.   The edges at higher levels belong to some $\fri^j$.
If $\fri^j$ is nonempty, then it contains
$2d^+(\rvi^j)$ edges of $\st(\rvi^j) \cup \lk(\rvi^j)$, and
$2d^-(\rvi^{j+1})$ edges of $\st(\rvi^{j+1}) \cup \lk(\rvi^{j+1})$, but these two sets overlap in two edges.
Thus, $|E(\fri^j)| \ge 2d^+(\rvi^j) + 2d^-(\rvi^{j+1})-2$.
Hence,
$$|E(\fri)| \ge
  2\nrvi + \sum_{j:\fri^j \ne \emptyset} \left(2d^+(\rvi^j) +
	2d^-(\rvi^{j+1})-2\right)
= 2\nrvi +
   \sum_{(j, \sigma) : d^\sigma(\rvi^j) > 0} (2d^\sigma(\rvi^j)-1) $$
where in the last sum $j \in [\nrvi]$ and $\sigma \in \{-, +\}$.

Therefore, when $\fri$ is nonempty,
\begin{equation}\label{eq:diff}
 2|E(\fri)| - \pS_i \geq
  \sum_{j=1}^\nrvi (4-2d_\br(\rvi^j))
  + d_\br(\rvi^1) + d_\br(\rvi^\nrvi)
  + \sum_{(j,\sigma): d^\sigma(\rvi^j) > 0}
	\left( (4-d_\br(\rvi^j)) d^\sigma(\rvi^j)-2 \right)
.\end{equation}
In the first sum, only terms with $j \in \{1, \mrvi, \nrvi\}$ can be nonzero.
In the final sum, only terms with $d_B(\rvi^j) > 2$, which requires $j=\mrvi$, can be negative.  (Fig.~\ref{fig:multiplicity5} shows a situation where we have a negative term in the final sum.)
 We consider several situations.

(i) If $\fri$ is empty or $1 = \mrvi = \nrvi$ then $\pS_i=0$, so $\pS_i \le 2|E(\fri)|$.

(ii) Suppose that $1 = \mrvi < \nrvi$ or $1 < \mrvi = \nrvi$; these situations are symmetric so we may assume $1 = \mrvi < \nrvi$.  Then $d_\br(\rvi^1)=3$ and $d_\br(\rvi^\nrvi) = 1$.  The only possible negative term in the final sum of \eqref{eq:diff} is for $(j,\sigma) = (1,+)$, which is at least $(4-3)(1)-2=-1$.  Hence $2|E(\fri)|-\pS_i \ge (4-2(3))+(4-2(1))+3+1+(-1) = 3$, and so $\pS_i \le 2|E(\fri)|$.

(iii) Suppose that $1 < \mrvi < \nrvi$.  Then $d_\br(\rvi^1)=d_\br(\rvi^\nrvi) = 1$ and $d_\br(\rvi^\mrvi) = 4$.  We may have two negative terms in the final sum of \eqref{eq:diff}, for $(j,\sigma) = (\mrvi, \pm)$. Each negative term is equal to $(4-4)d^\sigma(\rvi^j)-2 = -2$.  Therefore $2|E(\fri)|-\pS_i \ge 2(4-2(1)) + (4-2(4)) + 1+1 + 2(-2) = -2$, i.e., $\pS_i \le 2|E(\fri)|+2$.

In situations (i) and (ii), since each term $x_v x_w$ in $S_i$ is at most $x_\uz^2$, we get $S_i \le \pS_i x_\uz^2 \le 2|E(\fri)| x_\uz^2$.  In situation (iii), $x_\vi x_{q_i}$ and $x_\vipo x_{q_i}$ have coefficient at least $1$ in $S_i$, so we have
$$S_i
	\le
	  x_\vi x_{q_i} +x_\vipo x_{q_i} +
	  (\pS_i-2)x_\uz^2
	\le
	  x_\vi +x_\vipo +
	  2|E(\fri)|x_\uz^2
.$$
Thus, in all cases $S_i \le x_\vi +x_\vipo + 2|E(\fri)|x_\uz^2$.

Hence, using our estimates above and the Cauchy-Schwarz inequality,
\begin{align}
   2\lambda R
	&\le 2\lambda z + S_0 + \sum_{i=1}^{d-1} S_i
	\nonumber\\
	&\le 2\lambda z + 4\lambda + \sum_{i=1}^{d-1} \left(
		 x_\vi +x_\vipo + 2|E(\fri)|x_\uz^2 \right)
	\nonumber\\
	&\le 2\lambda z + 4\lambda +
		2\sum_{i=1}^d x_\vi +
		2(|E(G)|-(2d-1))x_\uz^2
	\nonumber\\
	&\le 2\lambda z + 4\lambda + 2\sqrt{dz} + (4n-4d-4)x_\uz^2.
		\label{eq:o2}
\end{align}

Substituting $\eqref{eq:o2}$ into \eqref{eq:o1}, it follows that when $n$ is large enough,
\begin{equation}\label{eq:o2-summary}
4n -16 \leq \lambda^3 
    \leq \lambda^2z \leq 4d + \lp 2\lambda z + 4\lambda + 2\sqrt{dz} + 4n-4d -4 \rp - \paren{\sqrt{d + R\lambda}-\sqrt{d}}^2.
\end{equation}

Cancelling terms and rearranging the inequality, we obtain that 
$$\paren{\sqrt{d+R\lambda}-\sqrt{d}}^2 \leq 2\lambda (z + 2) + 2\sqrt{dz} + 12,$$
which can be written as 
\begin{equation}\label{eq:o3}
\frac{(\lambda R)^2}{\paren{\sqrt{d+\lambda R}+\sqrt{d}}^2} \leq  2\lambda (z + 2)+ 2\sqrt{dz} + 12.
\end{equation}
From here, we want to give an upper bound on $\lambda R$. Note that from \eqref{eq:o2-summary}, we also have
\begin{align*}
\lambda^2 z & \leq 4d + (2\lambda z + 4\lambda + 2\sqrt{dz} + 4n -4d-4 )\\
            & \leq 4n + 2\lambda z + 4\lambda + 2\sqrt{dz}\\
            &\leq 4n + 2\lambda z + 4\lambda + 2z\sqrt{d},
\end{align*}
since $z\geq \lambda>1$.
Thus by the fact that $\lambda^3 \geq 4n-16$, we obtain that
\begin{equation}\label{eq:z}
z \leq \frac{4n+4\lambda}{\lambda^2-2\lambda-2\sqrt{d}} \leq \frac{\lambda^3+16+4\lambda}{\lambda^2-2\lambda-\sqrt{\lambda^3+16}} 
\leq \paren{1+o(1)}\lambda.
\end{equation}
Since $\lambda^3 \leq \lambda^2 z \leq 4n+ 2\lambda z + 4\lambda + 2\sqrt{dz}$, we also have
\begin{align*}
    4n &\geq \lambda^3 - 2\lambda z -4 \lambda -2\sqrt{d z}\\
    &\geq\lambda^3  - 2\lambda (1+o(1))\lambda -4 \lambda - \sqrt{(\lambda^3+16) (1+o(1))\lambda}\\
    &\geq \lambda^3 - 3(1+o(1))\lambda^2 -4\lambda\\
    &\geq \left(\lambda-(1+o(1)) \right)^3.
\end{align*}
Thus, we have
\begin{equation}\label{eq:lambda_bound}
\lambda\leq \sqrt[3]{4n}+(1+o(1)).
\end{equation}
Combining with Lemma \ref{lem:spec_wheel}, we get an asymptotic estimation of $\lambda$. 
\begin{equation*}
    \lambda=(1+o(1))\sqrt[3]{4n}.
\end{equation*}

Recall that $\lambda \leq z$. Hence, using \eqref{eq:z}, we have $z = (1+o(1))\lambda = (1+o(1))\sqrt[3]{4n}.$
Consequently we obtain from \eqref{eq:o2} that $\lambda R = O(n)$, which implies that $\paren{\sqrt{d+\lambda R}+\sqrt{d}}^2 = O(n)$. Now it follows from \eqref{eq:o3} that 
$$\lambda R = O\left(\sqrt{n\lambda z+ n\sqrt{dz}}\right) = O\left(\sqrt{n\lambda^2+n^{3/2}\lambda^{1/2}}\right)= O(n^{5/6}).$$
Substituting $\lambda R$ into \eqref{eq:o1} and using the fact that $\lambda^3 \geq 4n -16$, we obtain that
$$4n -16 \leq 4d + O(n^{5/6}),$$
which implies that $d\geq n-O(n^{5/6})$. This completes the proof of Claim \ref{cl:weak_big_degree}.
\end{proof}

\begin{proof}[Proof of Lemma \ref{lem:weak_big_degree}]
In order to further improve the lower bound on $d$ (as claimed in Lemma \ref{lem:weak_big_degree}), we need to give a non-trivial upper bound on $x_\uz^2 = \max_{v\neq \vz} x_v^2$.
We claim $x_\uz=O(n^{-1/3})$.

Let $\dd = d_G(\uz)$ and $\{u_1, u_2, \cdots, u_\dd\}$ be the neighbors of $\uz$ in $G$.  Since $G$ is outerplanar and has no $K_{2,3}$ subgraph, $\vz$ and $\uz$ have at most two common neighbors, so $\dd \le n+2-d = O(n^{5/6})$.
Most of the inequalities shown in Claim \ref{cl:weak_big_degree} hold in similar forms.
However, we have to treat any terms that involve $x_\vz$ separately from other terms, so our definitions of $\dR$ and $\dbr$ will be slightly different.

By the eigenequation for $x_\uz$, allowing for the possibility that some $\ui$ is $\vz$, and using $ab \le (a^2+b^2)/2$, we have
$$\lambda x_\uz^2
            = \dss_{i=1}^{\dd -1} x_\ui x_\uipo
        \leq 2x_\vz x_\uz + \dss_{u \in N(\uz)\bsh\{\vz\}} x_u^2
            = 2x_\uz + z'
$$
where we define $z' = \sum_{u \in N(\uz)\bsh\{\vz\}} x_u^2$.
Let $\dbr$ be the subgraph of $G$ consisting of the edges in $(\bigcup_{u \in N(\uz)\bsh\{\vz\}} \lk(u)) \bsh\st(\uz)$ and their endvertices.  Here $\dbr$ is similar in structure to $\br$, but $\dbr$ is missing the edges in $\lk(\ui)$ if some $\ui$ is $\vz$.
In a similar way to \eqref{eq:lz1} and \eqref{eq:lz2}, if we apply the eigenequations again and Cauchy-Schwartz, we have
\begin{equation*}
    \lambda z' \leq 2x_\uz + 2x_\uz\sqrt{\dd z'} + R',
\end{equation*}
where $R' = \sum_{vw \in E(\dbr)} x_{v} x_{w}$.


It follows from the same logic as in \eqref{eq:z-algebra} that 
\begin{align*}
    z'
       & \leq \frac{4\dd x_\uz^2}{\lambda^2}
	+ \frac{2(R'+2x_\uz)}{\lambda}
		- \paren{\sqrt{\frac{\dd x_\uz^2} {\lambda^2} +
		\frac{R'+2x_\uz}{\lambda}}
	-\frac{\sqrt\dd x_\uz}{\lambda}}^2.
\end{align*}
Then
\begin{align}
                 \lambda^2 (z'+2x_\uz)
                   \leq & 4\dd x_\uz^2 + 2\lambda (R'+ 2x_\uz) -\lp \sqrt{\dd x_\uz^2 + \lambda (R'+2x_\uz)} - \sqrt{\dd }x_\uz \rp^2 + 2\lambda^2 x_\uz \nonumber \\
                   \leq & 4\dd x_\uz^2 + 2\lambda R' + (2\lambda^2+4\lambda) x_\uz. \nonumber
\end{align}
Hence we have 
\begin{equation}\label{eq:o4}
        (4n-16)x_\uz^2  \leq \lambda^3 x_\uz^2  \leq  \lambda^2 (z'+2x_\uz) \leq 4\dd x_\uz^2 + 2\lambda R' + (2\lambda^2+4\lambda) x_\uz.
\end{equation}

We will use an inequality like \eqref{eq:o2} to bound $2\lambda R'$.  Similarly to \eqref{eq:tlr1}, we have
\begin{align}
2\lambda R'
  &\le \dss_{vw\in E(\dbr)} \lambda (x_{v}^2 + x_{w}^2)
    = \dss_{u \in V(\dbr)} d_\dbr(u)\, \lambda x_u^2
	\nonumber\\
  &= \dss_{u \in V(\dbr)\cap N_1(\uz) } d_\dbr(u)\, \lambda x_u^2 
    + \dss_{u \in V(\br)\cap N_2(\uz) } d_\dbr(u)\, \lambda x_u^2
	\nonumber\\
  &\le 2\dss_{i=1}^\dd \lambda x_\ui^2 
    + \dss_{u \in V(\br)\cap N_2(\uz) } d_\dbr(u)\,\lambda x_u^2 
	\nonumber\\
  &\le 2\lambda \left(x_\vz^2 + \sum_{u \in N(\uz)\bsh\{\vz\}} x_u^2\right)
    + \dss_{u \in V(\dbr)\cap N_2(\uz) } d_\dbr(u)\,
	\sum_{vw \in \lk(u)} x_v x_w.
	\nonumber\\
  &= 2\lambda (z'+1)
    + \dss_{u \in V(\dbr)\cap N_2(\uz) } d_\dbr(u)\,
	\sum_{vw \in \lk(u)} x_v x_w.
	\label{eq:Rprime}
\end{align}
For any vertex $v$ the terms containing $x_v$ in the sum above are
\begin{equation}\label{eq:vwcoeff}
\sum_{u \in N(v) \cap N_2(\uz) \cap V(\dbr)}\; \sum_{w \in \lk\iv(uv)} d_\dbr(u) x_v x_w .
\end{equation}
For each $u$ there are at most two vertices $w \in \lk\iv(uv)$.  

We will break the sum of \eqref{eq:Rprime} into four parts: $S_0'$ from $vw \in \lk(\uz)$, $S_1'$ from $vw \notin \lk(\uz)$ but $vw \in \lk(\vz)$, $S_2'$ from $vw \notin \lk(\uz)$ but $vw \in \st(\vz)$, and $S_3'$ from all remaining terms.

We can bound $S_0'$ in a similar way to \eqref{eq:mul1}, as
$$S_0' = \sum_{u \in V(\dbr)\cap N_2(\uz) } d_\dbr(u)\,
  \sum_{vw \in \lk(u)\cap\lk(\uz)} x_v x_w
	\le 4\lambda x_\uz^2 .$$

For $S_1', S_2', S_3'$ we use the following analysis.  If $v \notin N[\uz]$ then \eqref{eq:vwcoeff} has at most two vertices $u$, both of which belong to the same subgraph $\dbri = \dbr \cap \far(\ui \uipo, \uz)$.  Thus, one $u$ has degree at most $4$ in $\dbr$, and the other has degree at most $2$.  Since there are at most two vertices $w$ for each $u$, for $v \notin N[\uz]$ the sum of the coefficients of all terms $x_v x_w$ is at most $12$.  Moreover, if we take any edge $vw$ at level $1\hf$ or higher relative to $\uz$, then one endvertex $v$ satisfies $v \notin N[\uz]$, and $x_v x_w$ occurs in \eqref{eq:vwcoeff} with a total coefficient of at most $6$.

Hence, using the eigenequation for $x_\vz$, we can bound $S_1'$ as
$$S_1' = \sum_{u \in V(\dbr)\cap N_2(\uz) } d_\dbr(u)\,
  \sum_{vw \in (\lk(u)\cap\lk(\vz))\bsh\lk(\uz)} x_v x_w
	\le 6 \sum_{vw \in \lk(\vz)} x_v x_w
	= 6\lambda x_{\vz}^2 = 6\lambda.$$

Consider terms in the sum of \eqref{eq:Rprime} with $vw \notin \lk(\uz)$ and $vw \in \st(\vz)$, i.e., $v = \vz$.  If $v = \vz \notin N[\uz]$ then the total coefficient of $x_v x_w$ is at most $12$, as described above.  If $v=\vz \in N[\uz]$ then $\vz = \ui$ for some $i$, and the only possible vertices $u$ in \eqref{eq:vwcoeff} are $q'_{i-1} \in \lk\iv(u_{i-1} \ui)$ and $q'_i \in \lk\iv(\ui \uipo)$, which both have degree at most $2$ in $\dbr$, and for each such $u$ there is only one $w$ such that $vw \notin \lk(\uz)$.  Thus, the total coefficient of $x_v x_w$ is at most $4$.  In either case,
\begin{equation*}
S_2' = \sum_{u \in V(\dbr)\cap N_2(\uz) } d_\dbr(u)\,
  \sum_{vw \in (\lk(u)\cap\st(\vz))\bsh\lk(\uz)} x_v x_w
	\le 12 x_\vz x_\uz = O(x_\uz).
\end{equation*}

Finally, the terms in the sum of \eqref{eq:Rprime} with $vw \notin \lk(\uz) \cup \lk(\vz) \cup \st(\vz)$ give
\begin{align*}
S_3' &= \sum_{u \in V(\dbr)\cap N_2(\uz) } d_\dbr(u)\,
  \sum_{vw \in \lk(u)\bsh(\lk(\uz)\cup\lk(\vz) \cup\st(\vz))} x_v x_w \\
   &\le 6|E(G)\bsh(\lk(\vz)\cup\st(\vz))| x_\uz^2
	= 6(2n-2d-2)x_\uz^2
	= O(n^{5/6})x_\uz^2.
\end{align*}

Therefore, using the fact that $\lambda = O(n^{1/3})$,
\begin{align}
2 \lambda R' &\le 2\lambda(z'+1) + S_0' + S_1' + S_2' + S_3'
	\nonumber \\
  &= 2\lambda (z'+1) + 4\lambda x_\uz^2 + 6\lambda + O(x_\uz) +
		O(n^{5/6}) x_\uz^2
	\nonumber \\
  & = 2\lambda z' + 8\lambda + O(x_\uz) +
	O(n^{5/6}) x_\uz^2.\label{eq:o5}
\end{align}
Substituting \eqref{eq:o5} into \eqref{eq:o4}, and using $\dd = O(n^{5/6})$, we have
\begin{align}
    (4n-16) x_\uz^2  \leq \lambda^2 (z'+2x_\uz) & \leq 4\dd x_\uz^2 + 2\lambda R' + (2\lambda^2+4\lambda) x_\uz \nonumber\\
                      & \leq 4\dd x_\uz^2 + \lp 2\lambda z' + 8\lambda + O(x_\uz) + O(n^{5/6}) x_\uz^2  \rp + (2\lambda^2+4\lambda) x_\uz \nonumber\\
                      &\leq 2\lambda z' + O(n^{5/6}) x_\uz^2 + 8\lambda + (2\lambda^2 +4\lambda +O(1)) x_\uz.
                      \label{eq:o5a}
\end{align}
Rearranging the inequality in \eqref{eq:o5a}, we first obtain an upper bound on $z'$:
\[z' \leq \frac{O(n^{5/6}) x_\uz^2 + (4\lambda+O(1)) x_\uz+8\lambda }{ \lambda^2-2\lambda} =  O \lp n^{1/6} x_\uz^2 + \frac{4x_\uz}{\lambda} +\frac{1}{\lambda}\rp.\]
Now using the upper bound on $z'$ and \eqref{eq:o5a}, we have the following inequality:
\begin{align*}
    (4n-16) x_\uz^2  & \leq 2\lambda z' + O(n^{5/6}) x_\uz^2 + 8\lambda + (2\lambda^2 +4\lambda +O(1)) x_\uz\\
                       & = O\lp n^{5/6} x_\uz^2  + \lambda^2 x_\uz + \lambda\rp.  
\end{align*}
It follows from the fact that $\lambda = O(n^{1/3})$ that 
\begin{equation*}
    x_\uz = O(n^{-1/3}).
\end{equation*}
Now using the bound $x_\uz = O(n^{-1/3})$ in \eqref{eq:o2}, we obtain a better bound on $d = d_G(\vz)$ in Claim \ref{cl:weak_big_degree}: 
\begin{equation*}
    4n-16 \leq \lambda^3 \leq 4d + 2\lambda z + 4\lambda + 2\sqrt{dz} +
	4(n-d)\, O(n^{-2/3}), 
\end{equation*}
which gives us
\begin{align*}
(4n-4d)(1-O(n^{-2/3})) &\le 16 + 2\lambda z + 4\lambda + 2\sqrt{dz} \\
	&= O(1) + O(n^{2/3}) + O(n^{1/3}) + O(\sqrt{n\cdot n^{1/3}})
		= O(n^{2/3})
\end{align*}
and thus $d\geq n-O(n^{2/3})$.
This completes the proof of Lemma \ref{lem:weak_big_degree}.
\end{proof}

\begin{lemma}\label{lem:v1-degree1}
$d_{\cH}(v_1) = 1$. Moreover, $x_{v_2} \geq x_{v_1}$.
\end{lemma}
\begin{proof}

Assume for the sake of contradiction that $d_{\cH}(v_1) \geq 2$. Recall that $\cH$ is edge-maximal. It follows that there must exist a vertex $q_1 \ne \vz$ such that $\{v_1, v_2, q_1\}$ is a hyperedge, and $q_1$ is a vertex of $\fr_1 = \far(v_1 v_2, v_0)$.
Let $\fr_1'$ be $\fr_1$ but with $v_2$ renamed as $v_0$.  Then $G' = G-(V(\fr_1)-\{v_1, v_2\}) \cup \fr_1'$
is outerplanar (we find the outerplanar embedding by flipping $\fr_1'$ over, i.e., reflecting it), and $G'$ is the shadow of a $3$-uniform hypergraph $\cH'$ that can be obtained from $\cH$ by replacing each hyperedge $\{v_2, u, w\}$ where $u, w \in V(\fr_1)$ by a hyperedge $\{\vz, u, w\}$.
Suppose $\bm{x}$ is the Perron-Frobenius eigenvector of $\cH$.
Since $x_\vz > x_{v_2}$ by Lemma \ref{lem:weak_big_degree}, it follows that
\begin{align*}
\ds\sum_{\{i_1,i_2,i_3\}\in E(\cH')}x_{i_1}x_{i_2} x_{i_3} - \ds\sum_{\{i_1,i_2,i_3\}\in E(\cH)}x_{i_1}x_{i_2} x_{i_3} & \geq x_{v_1} x_{q_1} (x_\vz-x_{v_2}) > 0.
\end{align*}
This implies that $\lambda(\cH') > \lambda(\cH)$, which contradicts $\cH$ attaining the maximum spectral radius. 

It remains to show that $x_{v_2} \geq x_{v_1}$. If $x_{v_2} < x_{v_1}$, then let $\bm{x'}$ be obtained from $\bm{x}$ by setting $x'_{v_1} = x_{v_2}$, $x'_{v_2} = x_{v_1}$ and keeping every other entry the same. Since $d_{\cH}(v_1) = 1$, it follows that $P_{\cH}(\bm{x'}) > P_{\cH}(\bm{x})$, which contradicts $\bm{x}$ being the Perron-Frobenius eigenvector of $\cH$.
\end{proof}

Now we are ready to show Theorem \ref{thm:outer}.

\begin{proof}[Proof of Theorem \ref{thm:outer}]
Let $\cH$ be an outerplanar $3$-uniform hypergraph on $n$ vertices with maximum spectral radius. Let $G$ be the shadow of $\cH$. Suppose the Perron–Frobenius eigenvector $\bm{x}$ of the adjacency tensor of $\cH$ is normalized so that the maximum eigenvector entry is $1$. Let $\vz$ be the vertex with the maximum eigenvector entry and $\{v_1, v_2, \cdots, v_d\}$ be the neighbors of $\vz$ in the clockwise order of the outerplanar drawing of $G$.

By Lemma \ref{lem:spec_wheel}, we have that $d(\vz)\geq n - O(n^{2/3})$ and for every other vertex $u \neq \vz$, $x_{u} = O(n^{-1/3})$.
Now we claim that $x_{v_1} = \Omega(n^{-1/3})$. By Lemma \ref{lem:v1-degree1}, we have that $d_{\cH}(v_1) =1$, i.e., $\{v_1, v_2, \vz\}$ is the unique hyperedge containing $v_1$. It follows by Lemma \ref{lem:v1-degree1} and the eigenequation for $x_{v_1}$ that 
$$\lambda x_{v_1}^2 = x_\vz x_{v_2} = x_{v_2} \geq x_{v_1}.$$
Together with \eqref{eq:lambda_bound}, this implies that 
$$x_{v_1} \geq \frac{1}{\lambda} = \Omega(n^{-1/3}).$$

Now we claim that for every vertex $u \in V(G)\backslash\{\vz\}$, $u$ is a neighbor of $\vz$ in $G$.
Suppose not.  The hyperedges incident with $\vz$ form a path in the dual of $G$ and there must be a hyperedge $\{w, s, t\}$ that is a leaf of the dual tree (excluding the outer face) but not an end of this path.
Then $w,s,t \ne \vz$ and one of these vertices, say $w$, has degree $2$ in $G$ and degree $1$ in $\cH$.
Now similarly to Lemma \ref{lem:v1-degree1}, consider the hypergraph $\cH'$ obtained from $\cH$ by by removing the hyperedge $\{w, s, t\}$ and adding the hyperedge $\{w, \vz, v_1\}$. It follows that 
\begin{align*}
\ds\sum_{\{i_1,i_2,i_3\}\in E(\cH')}x_{i_1}x_{i_2} x_{i_3} - \ds\sum_{\{i_1,i_2,i_3\}\in E(\cH)}x_{i_1}x_{i_2} x_{i_3} & \geq x_{w} x_\vz x_{v_1} - x_{w} x_{s} x_{t}.
\end{align*}
Note that $x_s x_t = O(n^{-2/3})$ while $x_\vz x_{v_1} = \Omega(n^{-1/3})$. It follows that $x_{w} x_\vz x_{v_1} > x_{w} x_{s} x_{t}$, which implies that $\lambda(\cH') > \lambda(\cH)$, contradicting $\cH$ being the extremal hypergraph of maximum spectral radius. Hence every vertex $u \in V(G)\backslash\{\vz\}$ is a neighbor of $\vz$ in $G$.

Again by the fact that $\cH$ attains the maximum spectral radius, it follows that $\cH$ is the unique $3$-uniform hypergraph $\cF_n$ with $K_1 + P_{n-1}$ as it shadow.
\end{proof}

\end{document}